\documentclass[12pt]{amsart}
\usepackage{amsmath,amssymb}
\newtheorem{theorem}{Theorem}
\newtheorem{proposition}[theorem]{Proposition}
\def\B{\Bbb B}
\def\C{\Bbb C}
\def\D{\Bbb D}
\def\vp{\varphi}
\def\dist{\mbox{dist}}
\def\ds{\displaystyle}

\title[Boundary behavior of the squeezing functions]
{Boundary behavior of the squeezing functions of $\C$-convex domains
and plane domains}

\author{Nikolai Nikolov and Lyubomir Andreev}

\address{N. Nikolov: Institute of Mathematics and Informatics\\Bulgarian Academy
of Sciences\\Acad. G. Bonchev 8, 1113 Sofia, Bulgaria\newline
\indent Faculty of Information Sciences\\
State University of Library Studies and Information Technologies\\
Shipchenski prohod 69A, 1574 Sofia, Bulgaria}\email{nik@math.bas.bg}

\address{L. Andreev: Institute of Mathematics and Informatics\\Bulgarian Academy
of Sciences\\Acad. G. Bonchev 8, 1113 Sofia, Bulgaria}
\email{lyubomir.andreev@math.bas.bg}



\thanks{The authors were partially supported by the Career Development Program
for Young Scientists at the Bulgarian Academy of Sciences, 2016--2017.}

\begin{document}

\begin{abstract} It is shown that any non-degenerate $\C$-convex domain
in $\C^n$ is uniformly squeezing. It is also found the precise behavior
of the squeezing function near a Dini-smooth boundary point of a plane domain.
\end{abstract}

\maketitle

Denote by $\B_n$ the unit ball in $\C^n.$ Let $M$ be an $n$-dimensional complex manifold,
and $p\in M.$ For any holomorphic embedding $f:M\to\B_n$ with $f(p)=0,$ set
$$s_M(f,p)=\sup\{r>0:r\B_n\subset f(M)\}.$$
The squeezing function of $M$ is defined by $\ds s_M(p)=\sup_f s_M(f,p)$
if such $f$'s exist, and $\ds s_M(p)=0$ otherwise.

If $\ds\inf_M s_M>0,$ then $M$ is said to be uniformly squeezing.

Many properties and applications of the squeezing function and the uniformly squeezing
manifolds have been explored by various authors, see e.g.~\cite{DGZ1,DGZ2,DF,DFW,FW,JK,KZ}.

By \cite[Theorem 2.1]{KZ}, any convex bounded domain in $\C^n$ is uniformly
squeezing. Our first aim is to extend this result to a larger class of domains.

A domain $D$ in $\C^n$ is called $\C$-convex if any non-empty intersection of $D$
with a complex line is a simply connected domain. Then $\C^n\setminus D$
is a union of hyperplanes (see e.g.~\cite[Theorem 2.3.9]{APS}). This easily implies that
if $D$ is degenerate, i.e.~containing complex lines, then $D$ is linearly equivalent to
$\C\times D',$ and hence $s_D=0.$

On the other hand, we have the following.

\begin{theorem} There exists a constant $c_n>0$ such that $s_D\ge c_n$ for any
non-degenerate $\C$-convex domain $D$ in $\C^n.$
\end{theorem}

\begin{proof} We shall use the idea of the proof of \cite[Theorem 1.1]{KZ}.

Let $p\in D.$ We may suppose that $p=0.$ It follows by \cite[Lemma 15]{NPZ} and the proof of
\cite[Theorem 13]{NPZ} that there exist an orthogonal map $T:\C^n\to\C^n$ and a
lower triangular $n\times n$ matrix $A$ with 1's on the main diagonal such that
$G=T(D)$ contains the 1-norm unit ball $E_n$ in $\C^n$ and all the coordinates of $Az$ are
different from 1 for any point $z\in G.$ It is easy to see that $||A||_{\max}\le 1$ and then
$||A^{-1}||_{\max}\le(n-1)!$. Denoting by $\D$ the unit disc,
it follows that the image $F_n=A(E_n)$ of $E_n$ under the linear map given by $A$ contains
$16d_n\D^n,$ where $d_n>0$ depends only on $n.$

Since the squeezing function is invariant under biholomorphisms, we may replace $D$ by $A(G).$ Note
that the orthogonal projection $D_j$ of $D$ onto the $j$-th coordinate complex line is a simply connected
domain (see e.g.~\cite[Theorem 2.3.6]{APS}). Let $\vp_j$ be a conformal map from $D_j$ onto
$\D$ with $\vp_j(0)=0.$ Since $\dist(0,\partial D_j)\le 1,$ then $4|\vp_j'(0)|\ge1$
by the K\"obe quarter theorem. Then the same theorem implies that $d_n\D\subset\vp_j(16d_n\D).$
So, for $\vp=(\vp_1,\dots,\vp_n)$ one has that
$$d_n\D^n\subset\vp(16d_n\D^n)\subset\vp(F_n)\subset\
\vp(D)\subset\varphi(\Pi_{j=1}^n D_j)\subset\D^n.$$
Hence $d_n\B_n\subset\vp(D)\subset\sqrt n\B_n$ which implies the desired result
with $c_n=d_n/\sqrt n.$
\end{proof}

As an application of Theorem 1, we shall prove briefly one of the main results in \cite{NPZ}
(whose original proof is close to that of \cite[Theorem 1.1]{KZ}). Denote by $\gamma_D,$ $\kappa_D$
and $\beta_D$ the Carath\'eodory, the Kobayashi and the Bergman metrics of $D.$ It is well-known
that $\gamma_D\le\kappa_D$ and $\gamma_D\le\beta_D$ (if $\beta_D$ is well-defined).
Moreover, if $D$ is a convex, resp.~$\C$-convex, domain, then
$\gamma_D=\kappa_D$ by the Lempert theorem, resp.~$\kappa_D\le4\gamma_D$
by \cite[Corollary 2]{NPZ}. Then Theorem 1 and the
estimate $$s_D^{n+1}\beta_D\le\sqrt{n+1}\kappa_D$$
(see e.g.~\cite[Theorem 3.1]{DF}) imply \cite[Theorem 12]{NPZ}:
\smallskip

{\it $\beta_D$ is comparable with $\gamma_D$ and $\kappa_D$ on any non-degenerate
$\C$-convex domain $D$ in $\C^n$ up to multiplicative constants depending only on $n.$}
\smallskip

Our second result is about the boundary behavior of the squeezing function
near a smooth boundary point of a plane domain.

By \cite[Theorem 5.3]{DGZ1}, resp. \cite[Theorem 1.3]{DGZ2}, if $D$ is a $\mathcal C^\infty$-smooth
bounded domain in $\C,$ resp. a strictly pseudoconvex domain in $\C^n, $ then
$$\lim_{z\to\partial D}s_D(z)=1.$$
Conversely, by \cite[Theorem 1.2]{Zim}, if the last holds for a smooth non-degenerate convex domain
$D$ in $\C^n,$ then $D$ is strictly pseudoconvex.

To refine \cite[Theorem 5.3]{DGZ1}, recall that a $\mathcal C^1$-smooth bounded domain $D$ in $\C^n$
is said to be Dini-smooth if the inner unit normal vector $n$ to $\partial D$ is Dini-conti\-nuous.
This means that $\ds\int_0^1\frac{\omega(t)}{t}dt<\infty,$ where
$\omega(t)=\sup\{|n_x-n_y|:|x-y|<t,\ x,y\in \partial D\}$ is the modulus
of continuity of $n.$

A boundary point $a$ of a domain $D$ in $\C^n$ is called Dini-smooth if
there exists a neighborhood $U$ of $a$ such that $D\cap U$ is a Dini-smooth domain.

It is clear that $\mathcal C^{1+\varepsilon}$-smoothness implies Dini-smoothness.

\begin{proposition} Let $D$ be a plane domain and $a\in\partial D.$

\noindent (a) If $a$ is Dini-smooth, then
$\ds\limsup_{z\to a}\frac{1-s_D(z)}{\delta_D(z)}<\infty.$

\noindent (b) If $a$ is $\mathcal C^1$-smooth, then
$\ds\lim_{z\to a}\frac{1-s_D(z)}{\delta_D(z)^\alpha}=0$ for any $\alpha<1.$

\end{proposition}

As usual, $\delta_D$ is the distance to $\partial D.$

Proposition 2(a) is inspired by the same inequality in the case of
$\mathcal C^4$-smooth strictly pseudoconvex domains, see \cite[Theorem 1.1]{FW}.

This result is optimal in two directions.

First, the inequality is sharp. Indeed, replacing \cite[Lemma 2.2]{DFW} by
\cite[Theorem 7]{NA} in the proof of \cite[Theorem 2.1]{DFW}, we obtain the following:
\smallskip

{\it If $a$ is a Dini-smooth boundary point of a domain $D$ in $\C^n$ which is not
biholomorphic to the unit ball, then $\ds\liminf_{z\to a}\frac{1-s_D(z)}{\delta_D(z)}>0.$}
\smallskip

Second, \cite[3. An example]{DFW} shows that Proposition 2(a) may fail in the
$\mathcal C^1$-smooth case.
\smallskip

\noindent{\it Proof of Proposition 2.}  (a) One may find a neighborhood $U$
of $a$ such that $E=U\setminus\overline D$ is a Dini-smooth domain.
Let $b\in E,$ $\ds\varphi(\zeta)=\frac{1}{\zeta-b},$ and
$F=\varphi(\C\setminus\overline E)\cup\{0\}.$
Let $\psi:F\to\D$ be a conformal map. The Dini-smoothness implies that $\psi$
extends to a $\mathcal C^1$-diffeomorphism from $\overline F$ to $\overline \D$
(see e.g. \cite[Theorem 3.5]{Pom}). Setting $\theta=\psi\circ\varphi,$ we may
assume that $\theta(a)=1$ and that there exists $r\in(0,2)$ such that
$\D_r\subset G:=\theta(D)\subset\D,$ where $\D_r=\{\zeta\in\D:|\zeta-1|<r.\}.$
Let $\zeta=|\zeta|e^{i\theta}\in\D_r$ such that $1-|\zeta|<r'<r$ and $|e^{i\theta}-1|<r-r'.$ If
$\ds f_\zeta(t)=\frac{t+\zeta}{1+\overline{\zeta}t}$ and
$\ds|t|<\rho(\zeta):=\frac{|\zeta|+r'-1}{1-(1-r')|\zeta|},$ then
$$|f_\zeta(t)-1|<r-r'+|f_\zeta(t)-e^{i\theta}|\le r-r'
+\frac{(1-|\zeta|)(1+|t|)}{1-|\zeta t|}<r.$$
Taking $f_\zeta^{-1}$ as a competitor in the definition of $s_G(\zeta),$ it follows that
$s_G(\zeta)\ge\rho(\zeta).$ This implies that
$$\limsup_{\zeta\to 1}\frac{1-s_G(\zeta)}{\delta_G(\zeta)}\le
\lim_{\zeta\to 1}\frac{1-\rho(\zeta)}{1-|\zeta|}=\frac{2-r'}{r'}.$$
Letting $r'\to r,$ and using that $s_D(z)=s_G(\theta(z))$ and
$$\ds\lim_{z\to a}\frac{\delta_G(\theta(z))}{\delta_D(z)}=|\theta'(a)|$$
it follows that
$$\limsup_{z\to a}\frac{1-s_D(z)}{\delta_D(z)}\le\frac{2-r}{r}|\theta'(a)|.$$

\noindent(b) We may proceed as above, having in mind that now $\psi$ extends
to a homeomorphism from $\overline F$ to $\overline \D$ which is H\"older
continuous with any exponent $\alpha<1$ (see e.g. \cite[Theorem 2]{Les}).
This easily implies that $\ds\lim_{z\to a}\frac{\delta_G(\theta(z))}{\delta_D(z)^\alpha}=0$
for any $\alpha<1.$\qed

\

\noindent\textsl{Acknowledgment.} The authors would like to thank the referee
for finding an essential gap in the original proof of Theorem 1.

\end{document}